\documentclass[11pt]{amsart}

\usepackage{geometry}
\usepackage{mathtools}
\usepackage{amsmath,amssymb,color, array, verbatim}
\usepackage{enumerate}
\usepackage{verbatim}

\newtheorem{thm}{Theorem}

\newtheorem{lem}[thm]{Lemma}

\newtheorem{prop}[thm]{Proposition}

\theoremstyle{definition}

\newtheorem{rem}[thm]{Remark}

\numberwithin{equation}{section}

\newcommand{\Z}{{\mathbb Z}}

\newcommand{\Tr}{{\mathrm{Tr}}}
\newcommand{\diag}{\mathrm{diag}}

\newcommand{\F}{\mathbb{F}}
\newcommand{\M}{\mathbb{M}}
\newcommand{\rvline}{\hspace*{-\arraycolsep}\vline\hspace*{-\arraycolsep}}

\begin{document}

\title{Nonderogatory matrices as sums of $p$-potent and nilpotent matrices}

\author[A. Pojar]{Andrada Pojar}
\address{Technical University of Cluj-Napoca, Department of Mathematics, Str. Memorandumului 28, 400114, Cluj-Napoca, Romania}
\email{andrada.pojar@math.utcluj.ro}

\begin{abstract}
We prove that every nonderogatory matrix $A$ over a field of positive odd characteristic $p,$ that is sum of a $p$-potent matrix and a nilpotent matrix, has a decomposition $A=E+V,$ such that $E^p=E,$ and $V^3=0.$
\end{abstract}



\maketitle

\section{Introduction}

Let $\F$ be a field with positive odd characteristic $p.$
As usual, the letter $\F_p=\mathbb{Z}_p$ will stand for the prime field of $p$ elements having characteristic $p$, for any positive integer $n$, the notation $\mathbb{M}_n(\F)$ will denote the full matrix ring of $n\times n$ matrices over $\F,$
 and $\mathbb{M}_{m,n}(\F)$ will denote the full matrix ring of  matrices over $\F,$ with $m$ rows and $n$ columns.

A square matrix is {\it nonderogatory} if its characteristic
and minimal polynomials coincide. A matrix is nonderogatory if and only if it is similar to a
companion matrix $C$.
We recall that a companion matrix $C\in M_n(\F)$ is a matrix of the form
$$C=C_{c_0,c_1,\ldots, c_{n-1}}=\left(\begin{array}{ccccc}
0 & 0  &\ldots & 0 & -c_0 \\
1 & 0  &\ldots & 0 & -c_1\\
\vdots & \vdots  &\cdots  & \vdots & \vdots \\
0 & 0 &\ldots & 1 & -c_{n-1}
 \end{array}\right).$$

A square matrix $E$ over $\F$ is $p$-potent if $E^p=E.$ A square matrix $A$ over $\F$ is nil-clean if there exist an idempotent $E$ and a nilpotent $N$ such that $A=E+N.$ We consider decompositions of nonderogatory matrices $A$ such that there exist a $p$-potent $E$ and a nilpotent $N,$ with $A=E+N.$

Nil-clean decompositions are related to clean ones, introduced by Nicholson in \cite{N}, when investigating exchange rings and were first studied by Diesl in \cite{Diesl}. An important result appeared in \cite{BCDM} and \cite{KLZ} about them is: every $n\times n$ matrix over a division ring $D$ is nil-clean if and only if $D=\F_2.$ A theme related to a set $S$ of nilpotents of a ring is the study of boundness of the nilpotence index of S--the existence of a positive integer $n$ such that $a^n=0$ for all $a\in S$ , and it was first studied first in \cite{KWZ}. \v{S}ter has proved in \cite{S} that there exists a nil-clean decomposition of every matrix of $\mathbb{M}_n(\F_2),$ with nilpotent $N$ having nilpotence index at most $4$ (i.e $N^4=0.$). Moreover this result has sharpness, that is there exist $4\times 4$ companion matrices over $\F_2,$ which cannot be decomposed as a sum of an idempotent and a nilpotent matrix of nilpotence index at most $3.$

An extension of nil-clean decompositions to finite fields of odd cardinality $q,$ was done in \cite{AM}: every matrix over such a field is the sum of a $q$-potent $E=E^q$ and a nilpotent. In \cite{B} there is even more -- the nilpotent $N$ involved in such a decomposition can be with nilpotence index at most $3$ (i.e. $N^3=0).$ Decompositions as a sum of a potent and square-zero matrix were considered in \cite{DGGL}.

In \cite{BM} it has been proved that if $\F$ is a field of positive characteristic, $p,$ then a companion matrix $A\in \mathbb{M}_n(\F)$ is nil-clean if and only if $A$ is nilpotent or unipotent or the trace of $A$ is of the form $t\cdot 1,$ with $t\in \{1,2,\dots,p\}$ such that $n>t$.

The characterization of companion matrix we are interested in is in the following Theorem:

\begin{thm}\cite{PComp}\label{thmPcomp}
Let $\F$ be a field with positive odd characteristic $p.$ Let $n\geq 1$ be a positive integer. Fix constants $c_0, c_1,\dots,c_{n-1}\in \F.$
Let $C$ be a companion matrix. Then $C$ is sum of a $p$-potent and a nilpotent if and only if  the trace of $C$ is an integer multiple of unity of $\F.$
\end{thm}
In \cite{BrMe} it has been proved that every nil-clean nonderogatory $n\times n$ matrix $A,$ over a field of positive characteristic, $p,$ with trace of $A$ not equal to $1,$ can be written as $A=E+V,$ such that $E^2=E,$ and $V^{p+1}=0.$ If $\Tr(A)=1,$ then there is a similar decomposition with $V^{p+2}=0.$ We prove that every nonderogatory matrix $A$ over a field of positive odd characteristic $p,$ that is sum of a $p$-potent and a nilpotent,
 has decomposition $A=E+V,$ such that $E^p=E,$ and $V^3=0.$ Moreover, this result has sharpness, that is there exists a $3\times 3$ companion matrix over a field of characteristic $3,$ which is sum of a tripotent and a nilpotent, but cannot be decomposed as a sum of a tripotent and a nilpotent matrix of nilpotence index at most $2.$

\section{Useful tools}

The following Proposition has been proved for $\F_p$ in \cite{Pojar}. The same proof is valid for any field with with positive odd characteristic $p.$

\begin{prop}\label{PrescrDecomp}
Let $\F$ be a field with positive odd characteristic $p.$
Let $n\geq 2$ and $k$ be positive integers such that $k<n.$ Let $a\in \{1,2,\dots,p-2\}.$ Fix constants $d_0, d_1,\dots,d_{n-1}\in \F$ and denote $D=C_{d_0,d_1,\ldots,d_{n-1}}+\diag (\underbrace{a,\dots,a }_{k \textrm{-times}}  ,0,\dots,0).$ For every polynomial $g\in \F[X]$ of degree at most $n-2$ there exist two matrices $E,$ $M$ in $\mathbb{M}_n(\F)$ such that:
\begin{enumerate}
\item $D=E+M$
\item $E^p=E$, and
\item $\chi_M=X^n+(k\cdot 1_{\F}+d_{n-1})X^{n-1}+g.$
\end{enumerate}
\end{prop}

The following lemma provides a tool  to obtain certain nonderogatory matrices.

\begin{lem}\label{use1}\cite{PComp}
 Let $k$ and $n$ be nonzero natural numbers, $n\geq k,$ and $\F$ be a field. Let $a_1,a_2,\dots,a_k$ be integer multiples of the unity. Every companion matrix $C\in \M_n(\mathbb{F})$ is similar to  $C'+\diag (a_1,a_2,\dots,a_k,0,\dots,0)$ for some companion matrix $C'\in \M_n(\mathbb{F}).$
\end{lem}

\begin{rem}\label{PUpperTr}
The transition matrix $P,$ from the canonical basis of $\F$ to the basis in the proof of Lemma \ref{use1},
is upper triangular, with all entries on the main diagonal equal to $1.$
\end{rem}

\begin{lem}\label{ppotent}
Let $n$ be a positive integer, $n>1,$ $p$ be an odd prime number, $k\in \{1,2,\ldots,n-1\},$ $a\in \{2,3,\dots, p-1\},$ $\F$ be a field with characteristic $p,$ and $C\in \mathcal{M}_n(\F)$ a companion matrix. If $\mbox{ }\Tr(C)=t\cdot 1,$ with $t \in \{1,2,\dots,p-1\},$
 then there exists a $p$-potent matrix of the form $$L=\left(\begin{array}{cc}
a\cdot I_k  & L_{1,2} \\
0 & 0
\end{array}\right)$$
such that $C-L$ is nilpotent.
\end{lem}
\begin{proof}
Let $C\in \mathcal{M}_n(\F)$ be a companion matrix, with $\Tr(C)=t.$ We are going to prove that there exist $k\in \{1,2,\dots,n-1\}$ and $a\in \{2,3,\dots,p-1\}$ such that $\Tr(C)=ka.$ If $n\leq p,$ then $t\neq 0,$ and  we choose $k\neq t,$ and $a=k^{-1}\cdot t$; if $p<n$ for $t\neq 0$ we make the same choice of $k,$ and $a;$ if $t=0,$ we choose $k=p,$ and $a=2.$

 By Lemma \ref{use1} there exists a companion matrix $C'$ such that
$$C=P(\diag (\underbrace{a-1,\dots,a-1}_{k \textrm{-times}} ,0,\dots,0)+C')P^{-1},$$ where $P$ is the transition matrix we stated in Remark \ref{PUpperTr} that it is upper triangular, with all entries on the main diagonal equal to $1.$ Since $\Tr(C')=(ka-k(a-1))\cdot 1=k\cdot 1,$ it follows by the proof of Proposition \ref{PrescrDecomp} that
there exists a $p$-potent matrix of the form $$U=\left(\begin{array}{cc}
a\cdot I_k  & U_{1,2} \\
0 & 0
\end{array}\right)$$
such that $\diag (\underbrace{a-1,\dots,a-1 }_{k \textrm{-times}} ,0,\dots,0)+C'-U=B,$ with $B$ nilpotent.
Therefore $C=P(U+B)P^{-1}.$ Since $PUP^{-1}$ has the form of $L=\left(\begin{array}{cc}
a\cdot I_k  & L_{1,2} \\
0 & 0
\end{array}\right),$ and because $PUP^{-1}$ is $p$-potent, $PBP^{-1}$ is nilpotent, the conclusion is true.
\end{proof}

\begin{lem}\label{anotherPpotent}
Let $p$ be an odd prime, $\F$ be a field of characteristic $p,$ $n$ a positive integer, $a$ an integer multiple of unity and $L$ an $n\times n$ matrix with entries in $\F,$ such that $L^p=L,$ then $(aI_n-L)^p=aI_n-L.$
\end{lem}
\begin{proof}
For a natural number $m>2,$ we denote as in \cite{Pojar}, by $\alpha_1(m)$-the smallest $\alpha\in \{1,2,\ldots,m\}$ such that $m\choose {\alpha}$ is a nonmultiple of $p.$ From \cite{Pojar}(Lemma II.4) it follows that $\alpha_1(p)=p,$ therefore
$$(aI_n-L)^p=(aI_n)^p-L^p=a^pI_n^p-L=aI_n-L.$$
\end{proof}
The proof of the following Lemma has a similar flow with Lemma $2$ in \cite{BrMe}, where $a=1$ and $p=2.$
\begin{lem}\label{mainLem}
Let $n\geq 3$ be a positive integer, and $p$ be an odd prime, $\F$ a field of characteristic $p,$  and $a\in \{1,2,\dots,p-1\}.$ Let $k\in \{1,2,\ldots,n-1\},$
such that $n-k$ is odd.
Let $C\in \mathcal{M}_n(\F)$ be a companion matrix. If $\mbox{trace}(C)=(\frac{n+k+1}{2}\cdot a-l)\cdot 1,$ for some $l\in \{1,2,\dots,p-1\},$
 then $C$ is similar to a matrix $D$ which has a decomposition $D=E+V,$ such that $E^p=E,$ $V^{k+1}=0,$ the first column of $V$ is $(0,0,\ldots,0),$ and the last line of $E$ is $(0,0,\ldots,0,a).$
\end{lem}
\begin{proof}
For $a=1$ by Lemma $2$ in \cite{BrMe} we get the decomposition with $E^2=E,$ and since an idempotent is a $p$-potent, it follows that we have the desired decomposition.
\\ Now let $a\in \{2,\ldots,p-1\}.$ Let $(e_1,e_2,\ldots,e_n)$ be the canonical basis of $\F^n$. We construct inductively a new basis $(f_1,f_2,\ldots,f_n)$ in the following way:
\\ $f_1=e_1;$
\\ for every even $i\in \{2,\ldots, n-k+1\},$ $f_i=Cf_{i-1}-af_{i-1};$
\\ for every odd $i\in \{2,\ldots, n-k+1\},$ $f_i=Cf_{i-1};$
\\ for every $i\in \{n-k+2,\ldots,n\},$ $f_i=af_{i-1}-Cf_{i-1};$

It can be proved that for every $j\in \{1,2,\ldots,n\}$  we have $\textrm{Lin}(\{e_1,\ldots,e_j\})=\textrm{Lin}(\{f_1,\ldots,f_j\}),$ hence $\{f_1,\ldots,f_n\}$ is a basis for $\F^n.$
The matrix associated to $C$ in the basis $\{f_1,\ldots,f_n\}$ has the form $$D=\left(\begin{array}{cc}
M & N \\
P & Q
\end{array}\right),$$
with $$M=\left(\begin{array}{cccccccc}
a  & 0 & 0 & 0 & \dots & 0 & 0 & 0 \\
1  & 0 & 0 & 0 & \dots & 0 & 0 & 0 \\
0  & 1 & a & 0 & \dots & 0 & 0 & 0 \\
0  & 0 & 1 & 0 & \dots & 0 & 0 & 0 \\
\vdots & \vdots & \vdots & \vdots & \ddots & \vdots & \vdots & \vdots \\
0  & 0 & 0 & 0 & \dots & a & 0 & 0 \\
0  & 0 & 0 & 0 & \dots & 1 & 0 & 0 \\
0  & 0 & 0 & 0 & \dots & 0 & 1 & a
\end{array}\right)\in \mathcal{M}_{n-k}(\F),$$

$$N=\left(\begin{array}{ccccc}
 0 & 0 & \dots & 0 & d_1 \\
 0 & 0 & \dots & 0 & d_2 \\
 \vdots & \vdots & \ddots & \vdots & \vdots \\
 0 & 0 & \dots & 0 & d_{n-k-1} \\
 0 & 0 & \dots & 0 & d_{n-k}
\end{array}\right)\in \mathcal{M}_{n-k,k}(\F),$$

$$P=\left(\begin{array}{ccccc}
 0 & 0 & \dots & 0 & 1 \\
 0 & 0 & \dots & 0 & 0 \\
 \vdots & \vdots & \ddots & \vdots & \vdots \\
 0 & 0 & \dots & 0 & 0 \\
 0 & 0 & \dots & 0 & 0
\end{array}\right)\in \mathcal{M}_{k,n-k}(\F),$$

and $Q=aI_k-C_{d_{n-k+1},\ldots,d_n}$ for some $d_{n-k+1},\ldots,d_n\in \F.$ Moreover from $$\mbox{trace}(M)+\mbox{trace}(Q)=\mbox{trace}(C),$$ it follows that
$$\frac{n-k+1}{2}\cdot a+ka+d_{n}=(\frac{n+k+1}{2}\cdot a-l)\cdot 1,$$ hence
$$\frac{n+k+1}{2}\cdot a+d_{n}=(\frac{n+k+1}{2}\cdot a-l)\cdot 1,$$ therefore
$$-d_{n}=l\in \{1,2,\dots,p-1\}.$$

By Lemma \ref{ppotent}, since $a\in \{2,3,\dots, p-1\},$ it follows that there exists a $p$-potent matrix of the form $$L=\left(\begin{array}{cc}
a\cdot I_l  & L_{1,2} \\
0 & 0
\end{array}\right)$$
such that $C_{d_{n-k+1},\ldots,d_n}-L$ is nilpotent.
We derive that there exists a decomposition $Q=(aI_k-L)+T$ such that $T$ is nilpotent.
Morevore by Lemma \ref{anotherPpotent} we know that $$K=aI_k-L$$ is $p$-potent.
We note for further use that the first column of $K$ is $(0,0,\ldots,0,0)$ and the last line of $K$ is $(0,0,\ldots,0,a).$
We consider $$V=\left(\begin{array}{cc}
R  & S \\
0 & T
\end{array}\right),$$
with
$$R=\left(\begin{array}{ccccccc}
0  & 0 & 0 & 0 & \dots & 0 & 0   \\
0  & 0 & 0 & 0 & \dots & 0 & 0   \\
0  & 1 & 0 & 0 & \dots & 0 & 0   \\
0  & 0 & 0 & 0 & \dots & 0 & 0  \\
\vdots & \vdots & \vdots &  \vdots & \ddots & \vdots & \vdots \\
0  & 0 & 0 & 0 & \dots & 0 & 0 \\
0  & 0 & 0 & 0 & \dots & 1 & 0
\end{array}\right)\in \mathcal{M}_{n-k}(\F),$$
(where $0$ and $1$ alternate on the diagonal below the main one),

$$S=\left(\begin{array}{ccccc}
 0 & 0 & \dots & 0 & d_1 \\
 0 & 0 & \dots & 0 & 0 \\
 0 & 0 & \dots & 0 & d_3 \\
 0 & 0 & \dots & 0 & 0 \\
 \vdots & \vdots & \ddots & \vdots & \vdots \\
 0 & 0 & \dots & 0 & d_{n-k}
\end{array}\right)\in \mathcal{M}_{n-k,k}(\F),$$
Since $RS=0$ it is easy to see that $V^{k+1}=\left(\begin{array}{cc}
R^{k+1}  & ST^k \\
O & T^{k+1}
\end{array}\right).$ But $T$ is a $k$ by $k$ nilpotent, and therefore $V^{k+1}=0.$
In order to see that $D-V$ is $p$-potent we observe that it has the form
$$D-V=\left(\begin{array}{cc}
X  & Y \\
Z & K
\end{array}\right),$$
with $$X=\left(\begin{array}{cccccccc}
a  & 0 & 0 & 0 & \dots & 0 & 0 & 0 \\
1  & 0 & 0 & 0 & \dots & 0 & 0 & 0 \\
0  & 0 & a & 0 & \dots & 0 & 0 & 0 \\
0  & 0 & 1 & 0 & \dots & 0 & 0 & 0 \\
\vdots & \vdots & \vdots & \vdots & \ddots & \vdots & \vdots & \vdots \\
0  & 0 & 0 & 0 & \dots & a & 0 & 0 \\
0  & 0 & 0 & 0 & \dots & 1 & 0 & 0 \\
0  & 0 & 0 & 0 & \dots & 0 & 0 & a
\end{array}\right)\in \mathcal{M}_{n-k}(\F),$$
$$Y=\left(\begin{array}{cccccc}
 0 & 0 & 0 & \dots & 0 & 0 \\
 0 & 0 & 0 & \dots & 0 & d_2 \\
 0 & 0 & 0 & \dots & 0 & 0 \\
 0 & 0 & 0 & \dots & 0 & d_4\\
 \vdots & \vdots & \vdots & \ddots & \vdots & \vdots \\
 0 & 0 & 0 & \dots & 0 & d_{n-k-1} \\
 0 & 0 & 0 & 0 & 0 & 0
\end{array}\right)\in \mathcal{M}_{n-k,k}(\F),$$
$$Z=\left(\begin{array}{ccccc}
 0 & 0 & \dots & 0 & 1 \\
 0 & 0 & \dots & 0 & 0 \\
 \vdots & \vdots & \ddots & \vdots & \vdots \\
 0 & 0 & \dots & 0 & 0 \\
 0 & 0 & \dots & 0 & 0
\end{array}\right)\in \mathcal{M}_{k,n-k}(\F),$$
and $$K=Q-T=(aI_k-L)+T-T=aI_k-\left(\begin{array}{cc}
aI_l  & L_{12} \\
0 & 0
\end{array}\right)=\left(\begin{array}{cc}
0  & -L_{12} \\
0 & aI_{k-l}
\end{array}\right)\in \mathcal{M}_{k}(\F).$$
By $0=YZ=XY=KZ=ZY$ we get that
$$(D-V)^p=\left(\begin{array}{cc}
X^p  & YK^{p-1} \\
ZX^{p-1} & K^p
\end{array}\right).$$

We will compute $X^{p-1}$ and $X^p$.
Since $a\in \{2,\dots, p-1\}$ we have that $a^{p-1}=1,$ and $a^p=a.$ Therefore
$$X^{p-1}=\left(\begin{array}{cccccccc}
1  & 0 & 0 & 0 & \dots & 0 & 0 & 0 \\
a^{p-2}  & 0 & 0 & 0 & \dots & 0 & 0 & 0 \\
0  & 0 & 1 & 0 & \dots & 0 & 0 & 0 \\
0  & 0 & a^{p-2} & 0 & \dots & 0 & 0 & 0 \\
\vdots & \vdots & \vdots & \vdots & \ddots & \vdots & \vdots & \vdots \\
0  & 0 & 0 & 0 & \dots & 1 & 0 & 0 \\
0  & 0 & 0 & 0 & \dots & a^{p-2} & 0 & 0 \\
0  & 0 & 0 & 0 & \dots & 0 & 0 & 1
\end{array}\right)\in \mathcal{M}_{n-k}(\F),$$
and hence $ZX^{p-1}=Z.$
We also obtain $X^p=X.$
Now we will compute $K^{p-1}$ and $K^p,$ for $K=\left(\begin{array}{cc}
0  & -L_{12} \\
0 & aI_{k-l}
\end{array}\right).$
Since $a\in \{2,3,\dots,p-1\}$ we obtain that $a^{p-1}=1,$ (and also $a^p=a)$.
Now it follows that $K^{p-1}=\left(\begin{array}{cc}
0  & -a^{-1}L_{12} \\
0 & I_{k-l}
\end{array}\right),$  $K^p=K,$ and also $YK^{p-1}=Y.$
Finally $(D-V)^p=D-V.$
For further use we point out that the last line of $D-V$ is $(0,0,\dots, 0, a),$ and the first column of $V$ consists only of zeros.
\end{proof}

\begin{lem}\label{maincor}
Let $\F$ be a field of positive odd characteristic $p.$ Suppose that $n\geq 3$ is a positive integer, $k$ a positive integer such that $n-k$ is odd, and $C\in \mathcal{M}_n(\F)$ is a companion matrix, such that $\Tr(C)=t\cdot 1,$ with $t\in \{0,1,\dots,p-1\},$
  and $\frac{n+k+1}{2}\cdot 1\neq 0.$  Then there exist $D,$ $E$ and $V$ such that $C\sim D=E+V$, $E^p=E,$ last line of $E$ is $(0,\dots,0,a)$ (with $a\in \{1,2,\dots,p-1\}$), $V^{k+1}=0$, and the first column of $V$ is $(0,\dots,0)$.
\end{lem}
\begin{proof} Let $\Tr(C)=t\cdot 1, t\in \{0,1,\dots,p-1\}.$
 There exist $l_{1},l_{2},\dots,l_{p-1},l_p\in \{0,1,\dots,p-1\}$ such that
             $$t\cdot 1=(\frac{n+k+1}{2}\cdot 1-l_{1})\cdot 1=(\frac{n+k+1}{2}\cdot 2-l_{2})\cdot 1=\dots$$ $$=(\frac{n+k+1}{2}\cdot(p-1)-l_{p-1})\cdot 1=(\frac{n+k+1}{2}\cdot p-l_p)\cdot 1$$
             We have the following equalties:
             $$l_{1}=(\frac{n+k+1}{2}\cdot 1-t)\cdot 1$$
             $$l_{2}=(\frac{n+k+1}{2}\cdot 2-t)\cdot 1$$
             $$\vdots$$
             $$l_{p-1}=(\frac{n+k+1}{2}\cdot (p-1)-t)\cdot 1$$
             $$l_p=(\frac{n+k+1}{2}\cdot p-t)\cdot 1.$$
          Since $l_{1},l_{2},\dots,l_{p-1},l_p$ are $2$ by $2$ different (since $\frac{n+k+1}{2}\cdot 1\neq 0)$, and $p\geq 3,$
          there exists $a\in \{1,2,\dots,p-1\},$ such that $l_a\in \{1,2,\dots,p-1\}.$ Hence,  by Lemma \ref{mainLem}, there exist the $p$-potent $E$ and the nilpotent $V,$ such that $C=E+V$ last line of $E$ is $(0,\dots,0,a)$, $V^{k+1}=0$, and the first column of $V$ is $(0,\dots,0)$.
\end{proof}

\begin{lem}\label{-3}
Let $\F$ be a field of positive odd characteristic $p.$ Suppose that $n\geq 3$ is a positive integer, $k\in \{1,2,\ldots,n-1\},$ such that $n-k$ is odd, and $C\in \mathcal{M}_n(\F)$ is a companion matrix. If any of the following conditions is true:
 \begin{enumerate}
 \item there exists $a\in \{1,2,\dots,p-1\}$ such that $\Tr(C)=(\frac{n-(k+1)}{2}\cdot a+l)\cdot 1,\mbox{ }l\in \{1,2,\dots,p-1\}$ ;
 \item  $\frac{n-(k+1)}{2}\cdot 1\neq 0,$ and $\Tr(C)\in \{0,1,2,3,\dots,p-1\},$
 \end{enumerate}
 then there exist $D,$ $E$ and $V$ such that $C\sim D=E+V$, $E^p=E,$ last line of $E$ is $(0,\dots,0)$, $V^{k+1}=0$, and the first column of $V$ is $(0,\dots,0)$.
\end{lem}
\begin{proof}
\begin{enumerate}
\item Assume that there exists $a\in \{1,2,\dots,p-1\}$ such that $\Tr(C)=(\frac{n-(k+1)}{2}\cdot a+l)\cdot 1,\mbox{ }l\in \{1,2,\dots,k-1\}.$
  Then $\Tr(aI_n-C)=(\frac{n+k+1}{2}\cdot a-l)\cdot 1.$ Since there exists a companion $C_1\in \M_n(\F)$ such that $-C\sim C_1,$ then $aI_n-C\sim aI_n+C_1.$
   \\By Lemma \ref{use1} there exists a companion $C'\in \M_n(F)$ such that $C_1\sim -aI_n+C'$. Then $aI_n-C\sim C'$, and we have $\Tr(C')=(\frac{n+k+1}{2}\cdot a-1)\cdot 1,$ and also $C\sim aI_n-C'.$
   \\By the former equality we get using Lemma \ref{mainLem} that there exist $D',$ $E'$ and $V'$ such that $C'\sim D'=E'+V'$, $E'^p=E',$ last line of $E'$ is $(0,\dots,0,a)$, $(V')^{k+1}=0$, and the first column of $V'$ is $(0,\dots,0).$
   \\We obtain that $C\sim (aI_n-E')+(-V'),$ and therefore there exist $D,$ $E=aI_n-E'$ and $V=-V'$ such that $C\sim D=E+V$, $E^p=E,$ last line of $E$ is $(0,\dots,0)$, $V^{k+1}=0$, and the first column of $V$ is $(0,\dots,0)$

\item Assume $\Tr(C)=t\cdot 1, t\in \{0,1,\dots,p-1\},$ and $\frac{n-(k+1)}{2}\cdot 1\neq 0.$
\\There exist $l_{1},l_{2},\dots,l_{p-1},l_p\in \{0,1,\dots,p-1\}$ such that
             $$t\cdot 1=(\frac{n-(k+1)}{2}\cdot1+l_{1})\cdot 1=(\frac{n-(k+1)}{2}\cdot2+l_{2})\cdot 1=\dots=$$ $$(\frac{n-(k+1)}{2}\cdot(p-1)+l_{p-1})\cdot 1=(\frac{n-(k+1)}{2}p+l_p)\cdot 1$$
             We have the following equalties:
             $$l_{1}=(t-\frac{n-(k+1)}{1}\cdot 1)\cdot 1$$
             $$l_{2}=(t-\frac{n-(k+1)}{2}\cdot 2)\cdot 1$$
             $$\vdots$$
             $$l_{p-1}=(t-\frac{n-(k+1)}{2}\cdot (p-1))\cdot 1$$
             $$l_p=(t-\frac{n+k+1}{2}\cdot p)\cdot 1.$$
          Since $l_{1},l_{2},\dots,l_{p-1},l_p$ are $2$ by $2$ different (since $\frac{n-(k+1)}{2}\cdot 1\neq 0)$, and $p\geq 3$, there exists $a\in \{1,2,\dots,p-1\}$ such that $l_a\in \{1,\dots,p-1\}.$ Hence condition $1)$ is fulfilled and hence there exist the $p$-potent $E$ and the nilpotent $V$ such that $C=E+V$ last line of $E$ is $(0,\dots,0,0)$, $V^{k+1}=0$, and the first column of $V$ is $(0,\dots,0)$.
\end{enumerate}
\end{proof}

\begin{rem}\label{Qdiag}
 The matrix $Q$-the transition matrix from the canonical basis of $\F^n$ to the basis in the proof of Lemma $2$, with $k=2,$ from \cite{BrMe} is upper triangular, with the main diagonal $(1,\dots,1,-1)$
\end{rem}

\begin{lem}\label{trip}
Let $Q$ be the matrix in Remark \ref{Qdiag},$\mbox{ }F=\diag(0,\dots,0,1),$ and $E_1$ the idempotent $D-V$ in Lemma $2$, with $n$ odd, and $k=2,$ from \cite{BrMe}.
Assume the characteristic of the field in which these matrices have all the entries is $3.$ Then the matrix
$$E=F+QE_1Q^{-1}$$ is tripotent.
\end{lem}
\begin{proof}
Since last line of $E_1$ is $(0,\dots,0,1)$
we have that
$$FQE_1Q^{-1}=-FE_1Q^{-1}=-FQ^{-1}=F.$$
Let $c_n(Q^{-1})=(q'_{1n},\dots,q'_{n-1,n},-1)$ be the last column of $Q^{-1}$, and $T=\left(\begin{array}{ccccc}
 0 & 0 & \dots & 0 & q'_{1n} \\
 0 & 0 & \dots & 0 & q'_{1n}-d_2 \\
 0 & 0 & \dots & 0 & q'_{3n} \\
 0 & 0 & \dots & 0 & q'_{3n}-d_4 \\
 \vdots & \vdots & \ddots & \vdots & \vdots \\
 0 & 0 & \dots & 0 & q'_{n-2,n} \\
 0 & 0 & \dots & 0 & q'_{n-2,n}-u \\
 0 & 0 & \dots & 0 & -1 \\

\end{array}\right).$ We get:
$$QE_1Q^{-1}F=QE_1(0 \dots 0\mbox{ } c_n(Q^{-1}))=QT$$
 .\\
We derive that $$E^2=F+QE_1Q^{-1}+F+QT=-F+QE_1Q^{-1}+QT,$$
and $E^3=-F+F+FQT-QT+QE_1Q^{-1}+QE_1T=F-QT+QE_1Q^{-1}+QE_1T=E+QE_1T-QT.$
\\Moreover $$I_n-E_1=\left(\begin{array}{cc}
X'  & Y' \\
Z' & K'
\end{array}\right),$$
with $$X'=\left(\begin{array}{cccccccc}
0  & 0 & 0 & 0 & \dots & 0 & 0 & 0 \\
-1  & 1 & 0 & 0 & \dots & 0 & 0 & 0 \\
0  & 0 & 0 & 0 & \dots & 0 & 0 & 0 \\
0  & 0 & -1 & 1 & \dots & 0 & 0 & 0 \\
\vdots & \vdots & \vdots & \vdots & \ddots & \vdots & \vdots & \vdots \\
0  & 0 & 0 & 0 & \dots & 0 & 0 & 0 \\
0  & 0 & 0 & 0 & \dots & -1 & 1 & 0 \\
0  & 0 & 0 & 0 & \dots & 0 & 0 & 0
\end{array}\right)\in \mathcal{M}_{n-k}(\F),$$ and
$$Y'=\left(\begin{array}{cccccc}
 0 & 0 & 0 & \dots & 0 & 0 \\
 0 & 0 & 0 & \dots & 0 & -d_2 \\
 0 & 0 & 0 & \dots & 0 & 0 \\
 0 & 0 & 0 & \dots & 0 & -d_4\\
 \vdots & \vdots & \vdots & \ddots & \vdots & \vdots \\
 0 & 0 & 0 & \dots & 0 & -d_{n-k-1} \\
 0 & 0 & 0 & 0 & 0 & 0
\end{array}\right)\in \mathcal{M}_{n-k,k}(\F),$$
$$Z'=\left(\begin{array}{ccccc}
 0 & 0 & \dots & 0 & -1 \\
 0 & 0 & \dots & 0 & 0 \\
 \vdots & \vdots & \ddots & \vdots & \vdots \\
 0 & 0 & \dots & 0 & 0 \\
 0 & 0 & \dots & 0 & 0
\end{array}\right)\in \mathcal{M}_{k,n-k}(\F),$$
and $$K'=\left(\begin{array}{cc}
1  & u \\
0 & 0
\end{array}\right)\in \M_2(\F).$$
$$(I_n-E_1)T=\left(\begin{array}{ccccc}
 0 & 0 & \dots & 0 & 0 \\
 0 & 0 & \dots & 0 & -q'_{1n}+q'_{1n}-d_2+d_2 \\
 0 & 0 & \dots & 0 & 0 \\
 0 & 0 & \dots & 0 & -q'_{3n}+q'_{3n}-d_4+d_4 \\
 \vdots & \vdots & \ddots & \vdots & \vdots \\
 0 & 0 & \dots & 0 & 0 \\
 0 & 0 & \dots & 0 & -q'_{n-2,n}+q'_{n-2,n}-u+u \\
 0 & 0 & \dots & 0 & 0 \\

\end{array}\right)=0$$

We have obtained $QT=QE_1T,$ and in conclusion $E^3=E.$

\end{proof}

\section{Main result}

\begin{thm}\label{MainThm}
Let $\F$ be a field of positive odd characteristic $p.$ Suppose that $n\geq 1$ is a positive integer and $A\in \mathcal{M}_n(\F)$ is a nonderogatory matrix, such that $\Tr(A)=t\cdot 1,$ with $t\in \{0,1,\dots,p-1\}.$
Then there exist matrices $E$ and $V$, such that  $A=E+V,$ with  $E^p=E$ and $V^3=0.$
\end{thm}
\begin{proof} It is sufficient to provide decomposition only when $A$ is a companion matrix (since a non-derogatory matrix is similar to a companion matrix, a matrix which is similar to an idempotent matrix is an idempotent matrix, while a matrix similar to a nilpotent is a nilpotent matrix).So take $C$ a companion matrix, with $\Tr(C)=t\cdot 1,$ with $t\in \{0,1,\dots,p-1\}.$ (This form of the trace of the companion matrix $C$ is provided by Theorem \ref{thmPcomp}) We have the following cases
\begin{itemize}
\item $n=1$; we know that C is sum of a $p$-potent and a nilpotent $V$,which is $0$. $0^3=0,$ and we get the desired decomposition.
\item $n=2$; we know that C is sum of a $p$-potent and a nilpotent V, which has the property  $V^2=0,$ and therefore $V^3=0,$ and we obtain the desired decomposition.
\item $n\geq 3$ and $n$ is odd; By Lemma \ref{use1} there exists a companion $C'\in \M_n(\F)$ such that that $C\sim I_n+C'$.\\ For further use we are going to prove more: there exist $D,E$ and $V$ such that $C\sim D=E+V,$ $E^p=E$, last line of $E$ is $(0,\dots,0,a)$, $a\in \{1,2,\dots,p-1\},$ $V^3=0$, first column of $V$ is $(0,\dots,0)$. We call such a decomposition, with $a$ not necessary non-zero, a special decomposition .
    \begin{itemize}
    \item $\frac{n+3}{2}\cdot1\neq0;$
          then by Lemma \ref{maincor} there exists $a\in \{1,2,\dots,p-1\}$ such that $C$ has a special decomposition, with last line of the $p$-potent $(0,\dots,0,a)$

    \item $\frac{n+3}{2}\cdot1=0;$
             \begin{itemize}
             \item $p\neq 3;$ in this situation we have $\frac{n-3}{2}\cdot1\neq 0.$
                     It follows by Lemma \ref{-3}(2) that $C'$ has a special decomposition, and last line of the $p$-potent $E'$ involved is $(0,\dots,0,0)$. Therefore $C\sim I_n+E'+V,$ and $C$ has a special decomposition, and last line of the $p$-potent $E=I_n+E'$ is $(0,\dots,0,1)$, and we have our desired decomposition for $C.$
             \item $p=3;$ in this situation we have $\frac{n-3}{2}\cdot1=0.$
                   \begin{itemize}
                   \item $\Tr(C)=1;$ then $\Tr(C')=(\frac{n-3}{2}\cdot 1+1)\cdot 1,$ and therefore $C'$ has a special decomposition, and last line of the $p$-potent $E'$ involved is $(0,\dots,0,0)$. Therefore $C\sim I_n+E'+V,$ and $C$ has a special decomposition, and last line of the $p$-potent $E=I_n+E'$ is $(0,\dots,0,1)$, and we have our desired decomposition for $C.$
                   \item $\Tr(C)=-1;$ then $\Tr(C)=(\frac{n+3}{2}\cdot 1-1)\cdot 1.$ It follows by Lemma \ref{mainLem} that $C$ has a special desired decomposition, with last line of $E$ being equal to $(0,\dots,0,1).$
                   \item $\Tr(C)=0;$ By Lemma \ref{use1} there exists a companion matrix $C_1\in \M_n(\F)$ such that $C\sim F+C_1,$ where $F=\diag(0,\dots,0,1).$ It follows that $\Tr(C_1)=-1=(\frac{n+3}{2}\cdot 1-1)\cdot 1,$ and therefore, by Lemma $2$ from \cite{BrMe} we obtain that there exist the idempotent $E_1,$ and the nilpotent $V_1,$ with $V_1^3=0$, where first column of $V_1$ is $(0,\dots, 0)$, such that $C_1=Q(E_1+V_1)Q^{-1},$ where $Q$ is the matrix from Remark \ref{Qdiag}. By the same Remark the last line of $Q$ is $(0,\dots,0,-1).$ Since last line of $E_1$ is $(0,\dots,0,1)$ we derive that last line of $QE_1Q^{-1}$ is $(0,\dots,0,1),$ and therefore last line of $E=F+QE_1Q^{-1}$ is $(0,\dots,0,2)$. Moreover $C\sim D=E+QV_1Q^{-1},$  the nilpotent $V=QV_1Q^{-1}$ has nilpotence index at most $3,$ and first column $(0,\dots,0),$ and $E$ is tripotent by Lemma \ref{trip}, it follows that we have a desired decomposition for $C.$
                   \end{itemize}
              \end{itemize}

\item $n\geq 3$ and $n$ is even;  We are going to use a technique similar with the one used for nil-clean decompositions in the proof of the main theorem from \cite{BrMe},
      obtaining a decomposition as sum of a $p$-potent and a nilpotent with nilpotence index at most $3$, for the companion matrix $C\in \mathcal{M}_n(\F),$  in the following way:
\end{itemize}
\end{itemize}
     we know that $C$ is similar to a matrix of the form
\[D'=
\begin{pmatrix}
  \begin{array}{c}
    0
  \end{array}
  & \rvline &
  \begin{array}{cccc}
    0  & \dots & 0 & d_1
\end{array}
   \\
\hline
  \begin{array}{c}
    1 \\
    0 \\
    \vdots\\
    0
  \end{array}
  & \rvline &
  \begin{matrix}
   \large{D}
  \end{matrix}
\end{pmatrix}
\]
where $D\in \mathcal{M}_{3}(\F)$ is the nonderogatory matrix $D$ we have found in previouse case, where we proved that there exists a special  decomposition for $D=E+V,$ and last line of $E$ is $(0,\dots,0,a)$ with $a\in \{1,2,\dots,p-1\}$, first column of $V$ is $(0,\dots,0),$ and $V^3=0.$
 Then since $2\leq n-1,$ $D'=E'+W,$ with $E'^p=E'$, (because $a\in \{1,2,\dots,p-1\}$), and
\[W=
\begin{pmatrix}
  \begin{array}{c}
    0
  \end{array}
  & \rvline &
 \begin{matrix}
   0 & \dots & 0 & 0
 \end{matrix}
 \\
\hline
  \begin{array}{c}
     1 \\
     0 \\
     \vdots\\
     0 \\
  \end{array}
  & \rvline &
  \begin{matrix}
   \large{V}
  \end{matrix}
\end{pmatrix}
\]

Since first column of $V$ is $(0,\dots,0)$ we have $W^3=0$. Therefore since $D'\sim C$, and a matrix similar to a $p$-potent is a $p$-potent and a matrix similar to a nilpotent is nilpotent, we have found a decomposition as sum of a $p$-potent and a nilpotent with nilpotence index at most $3$ for $C.$

\end{proof}

The following Proposition proves that in general we cannot decrease the bound
of the nilpotence degrees of the nilpotent matrices which can be used to obtain
decompositions as in Theorem \ref{MainThm}.

\begin{prop}
Let $\F$ be a field of characteristic $3.$
Let $C$ be a $3 \times 3$, companion matrix over $\F$ such that $\mbox{trace}(C)=1$, and $0,$ $1$ and $-1$ are not eigenvalues of $C.$  If $E$ and $N$ are $3 \times 3$ matrices
over $\F$ such that $C=E+N$, $E^3=E,$ and $N^3=0$ then $N^2\neq 0$.
\end{prop}
\begin{proof}
Let $C$ be a $3 \times 3$, companion matrix over $\F$ such that $\mbox{trace}(C)=1$, and $0,$ $1$ and $-1$ are not eigenvalues of $C,$ and assume there exist matrices $E$ and $N$ such that $A=E+N$, $E^3=E,$ and $N^2=0.$

Since $N^2=0,$ and $N$ is a $3$ by $3$ matrix it follows that there exist at least $2$ blocks in the Frobenius normal form of $N$, therefore that matrix has at least $2$ zero rows. Hence $\mbox{rank}(N)\leq 1.$

Since $\mbox{trace}(C)=1$ then $C$ is similar to one of the following matrices $$\mbox{diag}(1,0,0), \mbox{diag}(-1,-1,0), \diag(1,1,-1).$$

Since $0,$ $1$ and $-1$ are not eigenvalues of $C,$ then $C,$ $I_n+C,$ $I_n-C$ are invertible.

We derive that:
$$3=\mbox{rank}(C)\leq \mbox{rank}(E)+\mbox{rank}(N)\leq \mbox{rank}(E)+1$$ and therefore $\mbox{rank}(E)\geq 2.$ Hence $E$ can not be similar to $\mbox{diag}(1,0,0).$

$$3=\mbox{rank}(I_3-C)\leq \mbox{rank}(I_3-E)+\mbox{rank}(-N)\leq \mbox{rank}(I_3-E)+1$$ and therefore $\mbox{rank}(I_3-E)\geq 2.$ Hence $E$ can not be similar to $\mbox{diag}(1,1,-1).$

$$3=\mbox{rank}(I_3+C)\leq \mbox{rank}(I_3+E)+\mbox{rank}(N)\leq \mbox{rank}(I_3+E)+1$$ and therefore $\mbox{rank}(I_3+E)\geq 2.$ Hence $E$ can not be similar to $\mbox{diag}(-1,-1,0).$

In conclusion the initial assumption is false and the conclusion of the theorem is true.
\end{proof}

\end{document}